\documentclass[11pt]{amsart}%
\usepackage{amscd,amsmath,latexsym,amsthm,amsfonts,amssymb,graphicx,color,geometry,hyperref}
\usepackage{url}

\usepackage{placeins}
\usepackage[utf8]{inputenc}
\usepackage{amsmath}
\usepackage{enumitem}
\usepackage{dsfont}
\usepackage[scr=rsfs]{mathalpha}
\usepackage{hyperref}
\usepackage{pifont}
\usepackage{xcolor}
\usepackage{biblatex}
\usepackage{amsfonts}
\usepackage{bm} 
\usepackage{amssymb}
\usepackage{graphicx}
\usepackage[all]{xy}
\usepackage{indentfirst}
\usepackage{graphicx}
\usepackage{graphics}
\usepackage{pict2e}
\usepackage{epic}
\numberwithin{equation}{section}
\usepackage{dsfont}
\providecommand{\U}[1]{\protect\rule{.1in}{.1in}}
\theoremstyle{plain}
\newtheorem{thm}{Theorem}[section]
\newtheorem{lem}[thm]{Lemma}
\newtheorem{cor}[thm]{Corollary}

\newtheorem{prop}[thm]{Proposition}
\newtheorem{pps}{Proposition}

\newtheorem{teo}{Theorem}

\theoremstyle{definition}

\newtheorem{?}[thm]{Problem}

\theoremstyle{definition}
\newtheorem*{nt*}{Notation}

\newcommand{\im}{{\rm Im}\,}

\begin{document}
\title[]{Expansivity and shadowing for composition operators on Paley-Wiener spaces}

\author{C. F. \'{A}lvarez}

\address{Departamento de Matemáticas, Universidad del Atlántico, Puerto Colombia, Colombia}  
\email{cfalvarez@mail.uniatlantico.edu.co}

\author{O. R. Severiano$^{*}$}

\address{ IMECC, Universidade Estadual de Campinas, Campinas, Brazil}
\email{osmar.rrseveriano@gmail.com}
\thanks{* Corresponding author. \\O.R. Severiano is a postdoctoral fellow at the Programa de Matemática and is supported by UNICAMP (Programa de Pesquisador de Pós-Doutorado PPPD) }


	
\begin{abstract}

In this article, we study composition operators $C_{\phi}f=f\circ \phi$ on the full range of Paley-Wiener spaces $B^2_a$ for $a>0.$ Here we compute the spectrum and the spectral radius of $C_{\phi}$ on $B^2_a.$ The spectrum of $C_{\phi}$ and the location of the fixed point of $\phi$ play a fundamental role in our analysis to show that no Paley-Wiener space $B^2_a$ supports composition operators possessing the positive shadowing property. We also provide a complete classification of the composition operators on $B^2_a$ that are positively expansive and absolutely Cesàro bounded. As a consequence of these results we show that no Paley-Wiener space $B^2_a$ supports Li-Yorke chaotic composition operators.
\end{abstract}
	
\maketitle

\section{Introduction}
Let $H$ be a complex Hilbert space of complex-valued functions on a subset $\Omega.$  If $\phi$ is a self-map of $\Omega$, then the \textit{composition operator} with \textit{symbol} $\phi$ on the space $H$ is the linear operator defined by
\begin{align*}
C_{\phi}f=f\circ \phi,\quad f\in H.
\end{align*}
Operators of this type have been studied on a variety of function spaces, such as Dirichlet, Fock, $L^2$ and Newton spaces \cite{Bernardes-Darji,Cowen, Eva-Montes, Eva, Guo,Gordon}. The main direction of research in this framework relies on the comparison between the properties of $C_{\phi}$ with those of $\phi,$ and as is known in the literature, there are several issues in this regard. For instance, boundedness, compactness, norms, spectra, spectral radius, normality, complex symmetry, cyclicity, hypercyclicity, supercyclicity  and more recently notions such as expansivity (and its variants), Li-Yorke chaos and the positive shadowing property. In many cases, it is difficult to solve these problems completely, and some remain open.

In \cite{Chacon}, Chacón, Chacón, and Giménez initiated the study of composition operators on the classical Paley-Wiener Hilbert space $B^2_\pi.$ They proved that $C_\phi$ is a bounded composition operator on $B^2_\pi$ if and only if 
\begin{align}\label{symbol}
\phi(z)=cz+d\ \text{where} \ c\in \mathbb{R} \ \text{with} \ 0<|c|\leq 1 \ \text{and} \ d\in \mathbb{C},
\end{align}
 and as a consequence, showed that such operators are neither compact nor supercyclic. Similarly, on the Paley-Wiener Hilbert space $B^2_a$ for $a>0,$ the only bounded composition operators are those induced by symbols of the form \eqref{symbol}. In \cite{Hai}, Hai, Noor, and Severiano  studied bounded composition operators on the full range of Paley-Wiener spaces $B^2_a$, completely  characterizing which of these operators are self-adjoint, unitary, normal and complex symmetric. They also studied the phenomena of cyclicity for such operators, for instance, they extended \cite[Theorem 2.7]{Chacon}, showing that no Paley–Wiener space $B^2_a$ supports supercyclic bounded composition operators; and that the cyclicity of  $C_\phi$ occurs if and only if $c = 1$ and $d \in \mathbb{C}\setminus\mathbb{R}$ or 
$0 < |d| \leq \pi/a$ with $d \in \mathbb{R}$ \cite[Theorem 4]{Hai}.

Recently, researchers have begun to investigate other properties of linear dynamics in the setting of composition operators defined on spaces of holomorphic functions. \'Alvarez and Henr\'iquez-Amador \cite{Alvarez} and Blois and Severiano \cite{blois2025dynamics} completely characterized the linear fractional composition operators that are expansive (and its variants, as positively expansive and uniformly positively expansive), Li-Yorke chaotic and have the positive shadowing property on the Hardy space $H^2(\mathbb{C}_+)$ of the right half-plane  and on the weighted Bergman  space $\mathcal{A}_{\alpha}^2(\mathbb{C}_+)$ of the right half-plane, respectively.

In this article, we continue the study of bounded composition operators on Paley-Wiener spaces $B^2_a$. We show that in this setting, some properties can be completely determined, such as the spectrum  (Proposition \ref{spectrum}) and the spectral radius (Proposition \ref{radius}). We also show that no Paley-Wiener space $B^2_a$ supports bounded composition operators that are compact (Proposition \ref{compactness}), Li-Yorke chaotic (Proposition \ref{Li-Yorke chaotic}) and possess the positive shadowing property (Theorem \ref{shadowing}). We finish this article by characterizing all the bounded composition operators that are positively expansive (Theorem \ref{ThmExpansive}) and absolutely Cesàro bounded (Theorem \ref{abs}).

We summarize some of our main results in the following table. 

\FloatBarrier
\begin{table}[htb!]
\caption{Properties of composition operators on the Paley-Wiener space $B^2_a$} 
\centering
\begin{tabular}{|c|c|c|c|}
\hline
\textbf{Symbol $\phi(z)=cz+d$}  & \textbf{Spectrum of $C_{\phi}$}  & \textbf{Pos. exp. $C_{\phi}$ } & \textbf{Abs. Ces. bou. $C_{\phi}$ }\\
      \hline
      $c=-1$ and $d\in \mathbb{C}$   & $\{-1, 1\}$
      & \ding{55}
      & \ding{52}   \\[2mm]
      \hline 

$0<|c|<1$ and $d\in \mathbb{C}$& $\left\{ \lambda \in \mathbb{C}: |\lambda|\leq 1/\sqrt{|c|} 
\right\}$
      & \ding{52}
      & \ding{55} \\[2mm]
\hline

$c=1$ and $d\in \mathbb{R}$ & $\{e^{\mathrm{i}dt}:t\in [-a, a]\}$
      & \ding{55} 
      & \ding{52}  \\[2mm]
\hline

$c=1$ and $d\in \mathbb{C}\setminus\mathbb{R}$ & $\{e^{\mathrm{i}dt}:t\in [-a, a]\}$
      & \ding{55}  & \ding{55}   \\[2mm]
\hline

\end{tabular}
\end{table}
\FloatBarrier

\section{Preliminaries}\label{s1}
Throughout the paper, $\mathbb{N}, \mathbb{R}$ and $\mathbb{C}$ denote the set of all positive integers, the set of all real numbers and the set of all complex numbers, respectively. We write $\mathrm{Im}(z)$ to denote the imaginary part of the complex number $z$. If $f:X\to X$ is a function and $n$ is a positive integer,  then the $n$-th iterate of $f$ is 
 $f^{n}=f\circ f \circ\cdots\circ f$ ($n$ times). For convenience, we write $f^{0}$ to be the identity function. For $z\in X,$ the orbit of $z$ for $f$ is the set $\{z, f(z), f^2(z), f^3(z), \cdots\}.$

Let $E$ be a complex Banach space, the \textit{unit sphere} of $E$ is $S_{E}=\{x\in E:\|x\|=1\}.$ For simplicity, if $E=\mathbb{C}$, then we write $\mathbb{T}$ instead of $S_{\mathbb{C}}.$ By an \textit{operator} on $E,$ we mean a continuous linear map $T:E\to E.$ We denote by $L(E)$ the set of all operators on $E.$ Let $T\in L(E)$ for the next definitions. The norm of $T$ is $\|T\|=\sup\{\|Tx\|: x\in S_E\}.$ The \textit{spectrum}  of $T,$ denoted by $\sigma(T),$ is $\sigma(T)=\{  \lambda \in \mathbb{C}: T-\lambda I \ \text{is not invertible} \}.$  The \textit{spectral radius} of $T,$ denoted by $r(T),$ is determined by the spectral radius formula $r(T)=\lim_{n\to \infty}  \|T^n\|^{1/n}.$ We say that   $T_1\in L(E_1)$ and $T_2\in L(E_2)$ are (isometrically) \textit{similar} if there exists an  (isometric) invertible continuous linear map $U:X_1\to X_2$ such that $T_1=U^{-1}T_2U.$ This is very useful for our purposes, since several properties are invariant under similarity.

Let us now present the relevant definitions for our work. We say that $T\in L(E)$ is:
\begin{itemize}

\item  \textit{hyperbolic} if $\sigma(T)\cap \mathbb{T}=\emptyset.$

\item \textit{positively expansive} if for each $x\in S_{E}$ there exists  $n\in \mathbb{N}$ such that $\|T^nx\|\geq 2;$

\item \textit{Li-Yorke chaotic} if there exists $x\in E$ such that
\begin{align*}
\displaystyle\liminf_{n\to \infty}\|T^nx\|=0 \quad \text{and} 
\quad \displaystyle\limsup_{n\to \infty}\|T^nx\|=\infty.
\end{align*}

\item \textit{absolutely Cesàro bounded} if there exists a constant $C>0$ such that 
\begin{align*}
\sup\left\lbrace\frac{1}{n}\sum_{j=1}^n\|T^jx\|:n\in \mathbb{N}\right\rbrace\leq C\|x\|, \quad \text{for all} \ x\in E. 
\end{align*}
\end{itemize}
Let $\delta>0,$ then by a $\delta$-\textit{pseudo-orbit} of $T\in L(E)$, we mean a sequence $(x_n)_{n=1}^{\infty}$ in $E$ such that
\begin{align*}
\|Tx_n-x_{n+1}\|\leq \delta \quad \text{for all} \ n\in \mathbb{N}. 
\end{align*}
We say that $T\in L(E)$ has the \textit{positive shadowing property}  if for every $\varepsilon>0$ there exists $\delta>0$ such that every $\delta$-pseudo-orbit $(x_n)_{n =1}^{\infty}$ of $T$ is $\varepsilon$-shadowed  by a real orbit of $T,$ that is, there exists $ x \in E$ such that 
\begin{align*}
\|T^nx-x_n\|\leq \varepsilon, \quad \ \text{for all} \ n\in \mathbb{N}.
\end{align*} 
 We refer to \cite{Abadias, Bermudez,Bernardes-Bonilla, Bernardes} for further information on these concepts.

We now introduce the main framework of our study: \textit{Composition operators on Paley-Wiener spaces.} For $a>0,$ the Paley-Wiener space $B_a^2$  consists of all entire functions  of exponential type less than or equal to $a$  whose restriction to $\mathbb{R}$ belongs to $L^2(\mathbb{R}).$  Each space $B^2_a$ is a Hilbert space when endowed with the inner product
\begin{align*}
\langle f, g\rangle=\int^{\infty}_{-\infty}f(t)\overline{g(t)}dt.
\end{align*}
We shall use $\|\cdot\|_a$ to denote the corresponding norm. This norm makes $B^2_a$ a reproducing kernel Hilbert space, this means that for each $w\in \mathbb{C}$ there exists a function $k_w\in B^2_a,$ called \textit{reproducing kernel} for $B^2_a$ at $w,$ satisfying $f(w)=\langle f, k_w\rangle$ for all $f\in B^2_a.$ The function $k_w$ on $B^2_a$ is given by  
\begin{equation*}
k_w(z)=\frac{\sin a(z-\overline{w})}{\pi(z-\overline{w})},
\end{equation*}
\cite[Theorem 16]{Branges} and therefore its norm is 
\begin{align}\label{norm of kernel}
\|k_w\|_a^2=
\begin{cases}
\frac{2a}{\pi}, & \text{if} \ w\in \mathbb{R} \\
\frac{\sinh 2a \im(w)}{2\pi\mathrm{Im}(w)},& \text{if} \ w\in  \mathbb{C}\setminus\mathbb{R}.
\end{cases}
\end{align}
By Paley-Wiener theorem the sequence $(k_{\frac{\pi n}{a}})_{n=-\infty}^{\infty}$ constitutes a complete orthogonal set for $B^2_a$ (for instance see \cite{Levin}, for a proof in $B_{\pi}^2$).
 
 The only bounded composition operators on $B^2_a$ are those induced by affine symbols of the form 
\begin{align}\label{symbol1}
\phi(z)=cz+d, 
\end{align}
where $c\in \mathbb{R}$ with $0<|c|\leq 1$ and $d\in \mathbb{C}.$ For this reason, in the remainder of this work, the term \textit{composition operator on $B^2_a$} means that the inducing symbol of such an operator is as in \eqref{symbol1}. In \cite{Hai}, Hai, Noor and Severiano  showed that each composition operator on $B^2_a$ is isometrically similar to a weighted composition operator on $L^2[-a, a].$

\begin{pps}\label{equivalence operator} \normalfont (\cite[Proposition 1]{Hai}) If $\phi(z)=cz+d$ where $c\in \mathbb{R}$ with $0<|c|\leq 1$ and $d\in \mathbb{C},$ then the composition operator $C_{\phi}:B^2_a\to B^2_a$  is isometrically similar to the weighted composition operator $\widehat{C}_{\phi}: L^2[-a,a]\to L^2[-a,a]$ defined by 
\[(\widehat{C}_{\phi}F)(t)=\frac{1}{|c|}\chi_{(-|c|a, |c|a)}(t)e^{\frac{idt}{c}}F\left(\frac{t}{c}\right)\] 
where $\chi_{(p,q)}$ denotes a characteristic function. Moreover, we have $(\widehat{C}_{\phi}^*F)(t)=\overline{e^{idt}}F(ct).$
\end{pps}

In particular, this shows us that if $c=1$ then $C_{\phi}$ is isometrically similar to the multiplication operator $M_{\Psi},$ where $\Psi(t)=e^{idt}.$  As a consequence of Proposition \ref{equivalence operator}, Hai, Noor and Severiano characterized normal and unitary composition operators on each Paley-Wiener space $B^2_a,$ as follows.

\begin{pps}\label{normal-unitary} \normalfont (\cite[Proposition 3]{Hai}) Let $\phi(z)=cz+d$ where $c\in \mathbb{R}$ with $0<|c|\leq1$ and $d\in \mathbb{C}.$ Then 
\begin{itemize}
\item $C_{\phi}$ is normal on $B^2_a$ if and only if $c=1$ and $d\in \mathbb{C}$ or $c=-1$ and $d\in \mathbb{R},$
\item  $C_{\phi}$ is unitary on $B^2_a$ if and only if $c=\pm 1$ and $d\in \mathbb{R}.$
\end{itemize}
\end{pps}

Since $|c|\leq 1,$ the only automorphisms of the complex plane whose inverse also induce a composition operator on $B^2_a$ occur when $c=\pm 1.$

As we are interested in the behavior of the orbit of composition operators induced by symbols $\phi$ as in \eqref{symbol1}, let us compute the $n$-th iterate of such symbols. Inductively, we have 
\begin{align*}
\phi^{n}(z)= \begin{cases}z+n d, & \text{if} \ c=1, \\ 
c^{n} z+\frac{(1-c^n)d}{1-c},  & \text{if} \ c \neq 1 ,
\end{cases}
\end{align*}	
and $C_{\phi}^n=C_{\phi^{n}}.$ Therefore, if $0<|c|<1$ and $\omega:=d/(1-c),$ then $\omega$ is an attractive fixed point of $\phi$, that is,  $\phi^{n}(z)\rightarrow \omega$ pointwise as $n\rightarrow \infty$ for all $z\in \mathbb{C}.$

Now let us determine the canonical form of the symbols $\phi$ as in \eqref{symbol1}. We note that for $c\neq 1,$  we can write
 \begin{align}\label{comp}
\phi=\psi\circ \varphi\circ \psi^{-1}
\end{align}
 where $\varphi(z)=cz$ and $\psi(z)=z+d/(1-c).$ Since $C_{\psi}$ is an invertible operator on $B^2_a,$ it follows from \eqref{comp} that $C_{\phi}=C_{\psi}^{-1}C_{\varphi}C_{\psi}.$ Therefore, when investigating properties invariant under similarity, it suffices to consider the composition operators induced by the following two symbols $\varphi(z)=cz$ and $\psi(z)=z+d.$

In  \cite[Corollary~2.5]{Chacon}, Chac\'on, Chac\'on, and Gim\'enez show that there is no compact composition operator on $B^2_{\pi}.$ We finish this section by extending this result, showing that no Paley-Wiener space supports compact composition operators. Before proceeding, we need the following well-known result.

\begin{lem}\label{adjoint}Let $C_{\phi}$ be a composition operator on $B^2_a,$ then  $C_{\phi}^*k_w=k_{\phi(w)},$ for each $w\in \mathbb{C}.$ 
\end{lem}

In general, the lemma above plays an important role in the study of the compactness of composition operators in various reproducing kernel Hilbert spaces (\cite[Proposition 3.13]{Cowen}, \cite[Corollary 3.3]{Elliot} and \cite[page 43]{Shapiro}). 



\begin{prop}\label{compactness} There are no compact composition operators on $B^2_a.$
\end{prop}
\begin{proof}  Let $\phi(z)=cz+d$ where $c\in \mathbb{R}$ with $0<|c|\leq 1$ and $d\in \mathbb{C},$ then $\im \left(\phi \left(n\pi/a\right)\right)=\im (d)$ for all $n\in \mathbb{N}.$ Let $K_n$ denote the normalized reproducing kernel $k_{\frac{n\pi}{a}}.$ Then $(K_n)_{n=1}^{\infty}$ is an orthonormal sequence on $B^2_a$ and \eqref{norm of kernel} gives
\begin{align*}
\|C_{\phi}^*K_n\|_{a}^{2}&=\frac{\|k_{\phi(\frac{n\pi}{a})}\|^{2}_{a}}{\|k_{\frac{n\pi}{a}}\|^{2}_{a}}=\begin{cases}\displaystyle\frac{\operatorname{sinh}2a\mathrm{Im}(d)}{4a\mathrm{Im}(d)}, & \text{if}\ d\in \mathbb{C}\setminus\mathbb{R}, \\ 
1, & \text{if}\  d\in \mathbb{R}.
\end{cases}
\end{align*}
 Since $(K_n)_{n=1}^{\infty}$ tends weakly to zero and $C_{\phi}^*K_n$ does not converge to zero in norm, it follows that $C_{\phi}^*$ is not compact. Therefore, $C_{\phi}$ is also not compact.
\end{proof}

\section{Positive shadowing property of \texorpdfstring{$C_{\phi}$}{}}\label{s2}
In this section, our main goal is to classify all the composition operators on $B^2_a$ that possess the positive shadowing property. This is done thanks to the following result, which characterizes normal operators on Hilbert spaces that possess the positive shadowing property.

\begin{teo}\label{normal-hyperbolic}\normalfont (\cite[Theorem 30]{Bernardes}) Let $H$ be a Hilbert space and $T\in L(H)$ a normal operator. Then $T$ has the positive shadowing property if and only if $T$ is hyperbolic.
\end{teo}

For this reason, we begin this section by determining the spectra of  all the composition operators on the Paley-Wiener space $B^2_a.$

\begin{prop}\label{first estimates} Let $\phi(z)=cz$ where $c\in \mathbb{R}$ with $0<|c|\leq 1.$ Then  for all $f\in B^2_a,$ we have 
\begin{align}\label{isometry}
\|C_{\phi}f\|_a=(1/\sqrt{|c|})\|f\|_a.
\end{align}
In particular, the spectrum of $C_{\phi}$ is $\sigma(C_{\phi})=\left\{\lambda\in \mathbb{C}: |\lambda|\leq 1/\sqrt{|c|}\right\},$ when $0<|c|<1.$
\end{prop}
\begin{proof} Let $f\in B^2_a.$  If $0<c\leq1,$ then a  straightforward computation gives $\|C_{\phi}f\|_a=(1/\sqrt{c})\|f\|_a.$ Now if $-1\leq c<0,$ we use auxiliary symbols $\varphi$ and $\psi$ defined by $\varphi(z)=-cz$ and $\psi(z)=-z,$ and we obtain $\phi=\psi\circ \varphi.$  For what we have done so far and since $C_{\psi}$ is a unitary operator (see Proposition \ref{normal-unitary}), we get
\begin{align*}
\|C_{\phi}f\|_{a}=\|C_{\varphi}C_{\psi}f\|_{a}=(1/\sqrt{-c})\|C_{\psi}f\|_{a}=(1/\sqrt{-c})\|f\|_{a}.
\end{align*}
 In addition, if $0<|c|<1,$ then \eqref{isometry} shows us that $\sqrt{|c|}C_{\phi}$ is a non-unitary isometry. Since every non-unitary isometry has the closed unit disk as its spectrum (by Wold decomposition \cite{Neumann}), we are able to deduce that $\sigma(C_{\phi})=\left\{\lambda\in \mathbb{C}:|\lambda|\leq 1/\sqrt{|c|}\right\}.$
\end{proof}

As a consequence of Proposition \ref{first estimates}, we deal with the problem of identifying the composition operators on $B^2_a$ with closed range.

\begin{cor} Each composition operator on $B^2_a$ has closed range.
\end{cor}

\begin{proof} Let $\phi(z)=cz+d$ where $c\in \mathbb{R}$ with $0<|c|\leq 1$ and $d\in \mathbb{C}.$ If $c=1$ then $C_{\phi}$ is invertible and hence has closed range. If $c\neq 1,$ then $C_{\phi}$ is similar to $C_{\varphi}$ where $\varphi(z)=cz.$ Since closed range is preserved under similarity, it is enough to show that $C_{\varphi}$ has closed range. Let $g$ be a function in $B^2_a$ such that there exists $(f_n)_{n=1}^{\infty}$ in $B^2_a$ with $C_{\varphi}f_n\to g$ as $n\to \infty.$ By Proposition \ref{first estimates}, for all $n,m\in \mathbb{N},$ we have
\begin{align*}
\|f_n-f_m\|_a\leq \sqrt{|c|}\|C_{\varphi}(f_n-f_m)\|_a.
\end{align*}
which implies that $(f_n)_{n=1}^{\infty}$ is a Cauchy sequence in $B^2_a.$ Thus there exists $h\in B^2_a$ such that $f_n\to h$ as $n\to \infty,$ this gives $C_{\varphi}h=g$ and therefore $C_{\varphi}$ has closed range. 
\end{proof}

Now let us determine the spectrum of all composition operator on the Paley-Wiener space $B^2_a.$ This is line with the study of linear fractional composition operators, as is done, for example in \cite{Cowen, Gallardo-Rikka, Higdon}, where the spectrum of linear fractional composition operators have been completely determined on various weighted Dirichlet spaces of the unit disk, including the Hardy space $H^2(\mathbb{D})$ and the weighted Bergman space $\mathcal{A}^2_{\alpha}(\mathbb{D})$ for $\alpha>-1.$ 

Recall that, according to the Theory of Multiplication Operators \cite{Conway}, if $\Phi:[-a,a]\to \mathbb{C}$ is a continuous function then $M_{\Phi}$ is an operator on $L^2[-a,a]$ with $\sigma(M_{\Phi})=\Phi([-a,a])$ and $\|M_{\Phi}\|=\|\Phi\|_{\infty},$ where $\|\Phi\|_{\infty}=\sup\{|\Phi(t)|:t\in [-a,a]\}.$ This will be useful for us to study the operator $M_{\Psi},$ where $\Psi(t)=e^{idt}.$

\begin{prop}\label{spectrum} Let $\phi(z)=cz+d$ where $c\in \mathbb{R}$ with $0<|c|\leq 1$ and $d\in \mathbb{C}.$ Then the spectrum of $C_{\phi}$ acting on $B^2_a$ is

\begin{enumerate}[label=\textnormal{(\arabic*)}]
    \item \label{spec1} $\sigma(C_{\phi})=\{-1,1\},$ when $c=-1;$
    \item \label{spec2} $\sigma(C_{\phi})=\left\{ \lambda \in \mathbb{C}: |\lambda|\leq 1/\sqrt{|c|} 
\right\},$ when $0<|c|<1;$
    \item \label{spec3} $\sigma(C_{\phi})=\{e^{idt}:t\in [-a, a]\},$ when $c=1.$
\end{enumerate}
\end{prop}

\begin{proof} \ref{spec1} If $c=-1,$ then $\phi^{2}$ is the identity map and hence $C_{\phi}^2=I$ is the identity operator. Now observe that for $\lambda\in \mathbb{C},$ we have
\begin{align*}
(C_{\phi}-\lambda I)(C_{\phi}+\lambda I)=(1-\lambda^2)I.
\end{align*}
Therefore $C_{\phi}-\lambda I$ is invertible if and only if $\lambda \neq \pm1. $ This gives $\sigma(C_{\phi})=\{-1,1\}.$ For the remaining situations, we use the fact that the spectrum is invariant under similarity.
\ref{spec2} If $0<|c|<1,$ then $C_{\phi}$ is similar to  $C_{\varphi}$ where $\varphi(z)=cz.$ This combined with Proposition \ref{first estimates}, allows us to conclude that $\sigma(C_{\phi})=\left\{ \lambda \in \mathbb{C}: |\lambda|\leq 1/\sqrt{|c|} 
\right\}.$ \ref{spec3} If $c=1,$ then  $C_{\phi}$ is similar to the multiplication operator $M_{\Psi}$ on $L^2[-a,a]$ (see Proposition \ref{equivalence operator}). Thus it is enough to compute the spectrum of $M_{\Psi}.$ Since $\Psi$ is a continuous function on $[-a,a],$ it follows that $\sigma(C_{\phi})=\Psi([-a,a])=\{e^{idt}:t\in [-a,a]\}.$ This completes the proof.
\end{proof}

With this result in hand, we compute the spectral radius of all composition operators on the Paley-Wiener space $B^2_a.$

\begin{prop} \label{radius}
If $\phi(z)=cz+d$ where $c\in \mathbb{R}$ with $0<|c|\leq  1$ and $d\in \mathbb{C},$ then the spectral radius of $C_{\phi}$ on $B^2_a$ is 
\begin{enumerate}[label=\textnormal{(\arabic*)}]
\item \label{sr1}$r(C_{\phi})=e^{|\im(d)|a},$ when $c=1.$
\item \label{sr2} $r(C_{\phi})=1/\sqrt{|c|},$ when $c\neq 1;$
\end{enumerate}
\end{prop}
\begin{proof} \ref{sr1} If $c=1,$ then $C_{\phi}$ is normal (see Proposition \ref{normal-unitary}) and hence $r(C_{\phi})=\|C_{\phi}\|.$ On the other hand, since $\Psi$ is a continuous function in $[-a,a]$ and  $C_{\phi}$ is isometrically similar to the multiplication operator $M_{\Psi}$ on $L^2[-a,a],$ it follows that $\|C_{\phi}\|=\|M_{\Psi}\|=\|\Psi\|_{\infty}=e^{|\im(d)|a}.$ \ref{sr2} By Proposition  \ref{spectrum}, we have $(1/\sqrt{|c|})\leq \|C_{\phi}\|$ (including for $c=1$).  Now observe that $\phi=\psi\circ \varphi$ where $\varphi(z)=cz$ and $\psi(z)=z+d.$ By part \ref{sr1} and Proposition \ref{first estimates}, we obtain  $\|C_{\psi}\|=e^{|\im(d)|a}$ and $\|C_{\varphi}\|=1/\sqrt{|c|}.$ This allows us to obtain the following estimate
\begin{align}\label{estimates}
\frac{1}{\sqrt{|c|}}\leq \|C_{\phi}\|=\|C_{\varphi}C_{\psi}\|\leq \|C_{\varphi}\|\|C_{\psi}\|=\frac{e^{|\im(d)|a}}{\sqrt{|c|}}.
\end{align}
 If $c\neq1,$ then we apply the estimate \eqref{estimates} to the $n$-th iterate of $\phi$ and obtain
\begin{align}\label{radius espectral}
\frac{1}{\sqrt{|c|}}\leq \|C_{\phi}^n\|^{\frac {1}{n}}\leq \frac{  e^{   \frac{(1-c^n)|\mathrm{Im}(d)|a}{n(1-c)}   }         }{\sqrt{|c|}}.
\end{align}
Letting $n$ go to infinity in \eqref{radius espectral},  we get the equality $r(C_{\phi})=1/\sqrt{|c|}.$ 
\end{proof}

Finally, we show that no Paley-Wiener space $B^2_a$ supports a composition operator with the positive shadowing property. It is also worth mentioning that when the inducing symbol of the composition operator has a fixed point in the complex plane, we proceed as in the proof of \cite[Proposition~1]{Alvarez} to show that such operator does not have the positive shadowing property on $B^2_a$.

\begin{thm}\label{shadowing} No bounded composition operator on $B^2_a$ has the positive shadowing property.  
\end{thm}

\begin{proof}  Let $\phi(z)=cz+d$ where $c\in \mathbb{R}$ with $0<|c|\leq 1$ and $d\in \mathbb{C}.$ We first assume $c=1.$ In this case, by Propositions \ref{spectrum} and \ref{normal-unitary}, the composition operator $C_{\phi}$  is a normal operator that is not hyperbolic. Therefore, by Theorem \ref{normal-hyperbolic}, it follows that $C_{\phi}$ does not have the positive shadowing property. Now we assume $c\neq 1,$ then $\omega=d/(1-c)$ is the fixed point of $\phi$ in the complex plane.
Now we construct a $\delta$-pseudo-orbit of $C_{\phi}$ which cannot be $\varepsilon$-shadowed by any function in $B^2_a$ for any $\varepsilon>0.$ For this, we choose $f\in B^2_a$ such that $f(\omega)\neq 0$ and we put $h:=(1/\|f\|_{_a})f.$ Now, we observe that for each $\delta>0,$ the sequence $(f_n)_{n=1}^{\infty}$ defined as follows: $f_1=\delta h$ and
\begin{align*}
f_n=\delta\sum_{j=0}^{n-1}C_{\phi}^{j}h, \quad n\geq 2,
\end{align*}
is a $\delta$-pseudo-orbit of $C_{\phi}.$ Indeed, we can show inductively that $C_{\phi}f_n=f_{n+1}-f_1$ for all $n$ and hence $\|C_{\phi}f_n-f_{n+1}\|_a=\|f_1\|_a=\delta.$ Moreover, we have $f_1(\omega)=f(\omega)/\|f\|_a$ and 
\begin{align*}
f_n(\omega)=(\delta\sum_{j=0}^{n-1}C_{\phi}^{j}h)(\omega)=\delta\sum_{j=0}^{n-1}h(\phi^j(\omega))=n\frac{\delta f(\omega)}{\|f\|_a}, \quad n\geq2.
\end{align*}
Then by the Cauchy-Schwarz inequality, given $g\in B^2_a$ we have
\begin{align}\label{i1}
\|C_{\phi}^ng-f_n\|_{a}\|k_{\omega}\|_{a}&\geq |\langle C_{\phi}^ng-f_n,k_{\omega}\rangle|=|g(\omega)-n\frac{\delta f(\omega)}{\|f\|_{a}}|\geq n\frac{\delta| f(\omega)|}{\|f\|_{a}}-|g(\omega)|
\end{align}
 Letting $n\rightarrow \infty$ in \eqref{i1}, we obtain $\|C_{\phi}^ng-f_n\|_{a}\rightarrow \infty$ which shows that $(f_n)_{n=1}^{\infty}$ cannot be $\varepsilon$-shadowed for any $\varepsilon>0.$
\end{proof}
\section{Positive expansivity of \texorpdfstring{$C_{\phi}$}{}} 

In this section, we characterize all positively expansive composition operators on $B_a^2$. This will be done based on the following characterization obtained by Bernardes \textit{et al.} in \cite{Bernardes}.

\begin{teo}\label{positively} \normalfont  (\cite[Proposition~19]{Bernardes}) Let $T\in L(E).$ Then $T$ is positively expansive if and only if $\sup\{\|T^nx\|:n\in \mathbb{N}\}=\infty$ for every non-zero $x\in E.$ 
\end{teo}

In view of Theorem \ref{positively}, we begin this section by studying the behavior of the orbits of $C_{\phi}.$ First we consider the symbol $\phi(z)=-z+d$ where $d\in \mathbb{C}.$ Since $\phi^2$ is the identity map, it follows that for all $n,$  $\phi^{2n}$ is the identity map and $\phi^{2n+1}$ is the symbol $\phi$ itself. This allows us to conclude that the set $\{C^n_{\phi}f:n\in \mathbb{N}\}$ is finite for each $f\in B^2_a$ and hence $C_{\phi}$ is not positively expansive (by Theorem \ref{positively}). Now we consider $\phi(z)=z+d$ where $d\in \mathbb{C}.$ In addition, if $d\in \mathbb{R}$ then $C_{\phi}$ is unitary (see Proposition \ref{normal-unitary}) and hence $C_{\phi}$ is not positively expansive. To deal with the case $d\in\mathbb{C}\setminus\mathbb{R}$, we need the following auxiliary result (which we will also use later when studying absolute Ces\`aro boundedness).

\begin{lem}\label{witness lemma}  For $d\in\mathbb{C}\setminus\mathbb{R},$  consider the functions $F_1$ and $F_2$ defined on $[-a,a]$ by
\begin{align*}
F_1(t)=\chi_{[0,a]}(st) \quad \text{and} \quad F_2(t)=\chi_{[-a,0]}(st),
\end{align*}
where $s:=\operatorname{sgn}(\mathrm{Im}(d))\in\{-1,1\}.$ Then $F_1,F_2\in L^2[-a,a]$ and for every $n\in \mathbb{N}$ we have 
\begin{align}\label{non-expansive}
\|M_{e^{idt}}^nF_1\|_{L^2[-a,a]}^{2}=\frac{1-e^{-2n\,|\mathrm{Im}(d)|\,a}}{2n\,|\mathrm{Im}(d)|} \quad \text{and} \quad \|M_{e^{idt}}^nF_2\|_{L^2[-a,a]}^{2}=\frac{e^{2n\,|\mathrm{Im}(d)|\,a}-1}{2n\,|\mathrm{Im}(d)|}.
\end{align}
\end{lem}

\begin{proof}
It is immediate that $F_1,F_2\in L^2[-a,a].$ To show \eqref{non-expansive}, we assume  $s=1$ for the sake of simplicity.  In this case, we have
\begin{align*}
\|M^n_{e^{idt}}F_1\|_{L^2[-a,a]}^2&=\|M_{e^{indt}}F_1\|_{L^2[-a,a]}^2=\int^a_{-a}|e^{indt}F_1(t)|^2dt=\int^a_{-a}e^{-2n\mathrm{Im}(d)t}\chi_{[0,a]}(t)dt\\
&=\int_0^ae^{-2n\mathrm{Im}(d)t}dt=\frac{1-e^{-2n\mathrm{Im}(d)a}}{2n\mathrm{Im}(d)}
\end{align*}
and again a straightforward computation gives
\begin{align*}
\|M^n_{e^{idt}}F_2\|_{L^2[-a,a]}^2&=\|M_{e^{indt}}F_2\|_{L^2[-a,a]}^2=\int^a_{-a}|e^{indt}F_2(t)|^2dt=\int^a_{-a}e^{-2n\mathrm{Im}(d)t}\chi_{[-a,0]}(t)dt\\
&=\int^0_{-a}e^{-2n\mathrm{Im}(d)t}dt=\frac{e^{2n\mathrm{Im}(d)a}-1}{2n\mathrm{Im}(d)}.
\end{align*}
For $s=-1,$ we proceed similarly. This completes the proof.
\end{proof}

\begin{prop}\label{third estimate}
Let $\phi(z)=cz+d$ where $c=\pm1$ and $d\in\mathbb{C}$. Then $C_{\phi}$ is not positively expansive on $B^2_a$.
\end{prop}

\begin{proof}
By Theorem \ref{positively}, it is enough to exhibit a non-zero function $f\in B_a^2$ whose orbit for $C_{\phi}$ is bounded. The cases $c=-1$ and  $c=1$ with $d\in\mathbb{R}$ have already been studied. To deal with the remaining case, that is, $c=1$ and $d\in\mathbb{C}\setminus\mathbb{R}$ we consider an isometric isomorphism $U:L^2[-a,a]\to B^2_a$ such that $M_{e^{idt}}=U^{-1}C_{\phi}U$ (see Proposition \ref{equivalence operator}) and the function $F_1\in L^2[-a,a]$ given in Lemma \ref{witness lemma}. Then $UF_1$ is a non-zero function in $B^2_a$ and for all $n$ we have
\begin{align}\label{expansive}
\|C_{\phi}^nUF_1\|_a^2=\|UM_{e^{idt}}^nF_1\|_{a}^2=\|M_{e^{idt}}^nF_1\|_{L^2[-a,a]}^2=\frac{1-e^{-2n\,|\mathrm{Im}(d)|\,a}}{2n\,|\mathrm{Im}(d)|}.
\end{align}
Letting $n\to \infty$ in \eqref{expansive}, we obtain $\|C_{\phi}^nUF_1\|_a\to 0$ which implies that the orbit of $UF_1$ for  $C_{\phi}$ is bounded. Therefore $C_{\phi}$ is not positively expansive.
\end{proof}

The following proof follows the steps of the proof from \cite[Lemma 3.2]{blois2025dynamics} for the weighted Bergman space $\mathcal{A}^2_{\alpha}(\mathbb{C}_+)$ of the right half-plane, for $\alpha>-1.$

\begin{prop}\label{second estimates} If $\phi(z)=cz+d$ where $c\in \mathbb{R}$ with $0<|c|<1$ and $d\in \mathbb{C},$ then for each $f\in B^2_a$ there exists $\delta>0$ such that 
\begin{align}\label{inequality}
\|C_{\phi}^nf\|_{a}\geq (\delta/\sqrt{|c|^n})\|f\|_{a}
\end{align}
for all sufficiently large $n.$ In particular, $C_{\phi}$ is positively expansive on $B^2_a.$
\end{prop}
\begin{proof}
Let $\psi$ be the automorphism defined by $\psi(z)=z+d/(1-c).$ For each $n,$ we consider the auxiliary symbols $\varphi_n$ and $\psi_n$ defined as follows
\begin{align*}
\varphi_n(z)=c ^nz  \quad \text{and} \quad \psi_n(z)=z+\frac{1-c ^n}{1-c }d.
\end{align*}
Then $\phi^{n}=\psi_n\circ \varphi_n.$ Let $f\in B^2_a.$ If $f$ is the zero function then the result is immediate. So we can assume that $f$ is not the zero function. By Proposition \ref{first estimates}, we have
\begin{align*}
\|C_{\phi}^nf\|_{a}=\|C_{\varphi_n}C_{\psi_n}f\|_{a}= (1/\sqrt{|c|^n})\|C_{\psi_n}f\|_{a}.
\end{align*}
We next obtain a lower estimate for $\|C_{\psi_n}f\|_{a}.$ For each $w\in \mathbb{C},$ the Cauchy--Schwarz inequality gives
\begin{align*}
\|C_{\psi_n}f\|_{a}\|k_w\|_{a}\geq |\langle C_{\psi_n}f, k_w\rangle|=|\langle f, C_{\psi_n}^*k_w\rangle|=|\langle f, k_{\psi_n(w)}\rangle|=|f(\psi_n(w))|.
\end{align*}
Since $0<|c|<1,$ it follows that $\psi_n(w)\rightarrow \psi(w)$ as $n\rightarrow \infty,$ for all $w\in \mathbb{C}.$ In particular, $|f(\psi_n(w))|\rightarrow |f(\psi(w))|$ as $n\rightarrow \infty.$ Moreover, $\psi(\mathbb{C})=\mathbb{C}$ and $f$ is a non-zero holomorphic function, hence $f(\psi(\mathbb{C}))\neq\{0\}.$ Then we can choose $\beta\in \mathbb{C}$ with $f(\psi(\beta))\neq 0.$  In particular, for all sufficiently large $n,$ we have
$|f(\psi_n(\beta))|>|f(\psi(\beta))|/2.$ If $\delta=\frac{|f(\psi(\beta))|}{2\|k_{\beta}\|_{a}\|f\|_{a}},$ then
\begin{align}\label{expansiveI}
\|C_{\phi}^nf\|_a\geq \frac{1}{\sqrt{|c^n|}}\frac{|f(\psi_n(\beta))|}{\|k_{\beta}\|_a}>\frac{1}{\sqrt{|c^n|}}\frac{|f(\psi(\beta))|}{2\|k_{\beta}\|_a\|f\|_a}\|f\|_a=\frac{\delta}{\sqrt{|c^n|}}\|f\|_a
\end{align}
for all sufficiently large $n.$ In particular, since $1/\sqrt{|c|^n}\to \infty$ as $n\to \infty,$ the inequality \eqref{expansiveI} allows us to conclude that the orbits for $C_{\phi}$ are unbounded for each non-zero function in $B^2_a.$ Hence by  Theorem \ref{positively}, $C_{\phi}$ is positively expansive.
\end{proof}

We therefore obtain a characterization of all composition operators that are positively expansive on $B^2_a$ as follows.

\begin{thm}\label{ThmExpansive}
Let $\phi(z)=cz+d$ where $c\in \mathbb{R}$ with $0<|c|\leq 1$ and $d\in \mathbb{C}$. Then $C_{\phi}$ is positively expansive on $B^2_a$ if and only if $0<|c|<1$.
\end{thm}

Since positively expansive operators on Banach spaces cannot be Li-Yorke chaotic \cite[Theorem C]{Bernardes}, Theorem \ref{ThmExpansive} helps us prove that no Paley-Wiener space $B^2_a$ supports Li-Yorke chaotic composition operators.

\begin{prop}\label{Li-Yorke chaotic} No composition operator on $B^2_a$ is Li-Yorke chaotic.
\end{prop}
\begin{proof}
Let $\phi(z)=cz+d$ where $c\in \mathbb{R}$ with $0<|c|\leq 1$ and $d\in \mathbb{C}.$ If $c=-1,$ we have already seen that each orbit of $C_{\phi}$ forms a finite set for all $f\in B^2_a$, and hence $C_{\phi}$ cannot be Li-Yorke chaotic. If $c=1,$ then $C_{\phi}$ is normal (Proposition \ref{normal-unitary}). Since normal operators are not Li-Yorke chaotic \cite[Corollary 6]{Bermudez}, it follows that $C_{\phi}$ is not Li-Yorke chaotic. If $0<|c|<1,$ then $C_{\phi}$ is positively expansive by Theorem \ref{ThmExpansive} and hence it cannot be Li-Yorke chaotic according to \cite[Theorem C]{Bernardes}.
\end{proof}

\begin{cor} No composition operator on $B^2_a$ is hypercyclic.
\end{cor}

The corollary above can also be obtained from the fact that no Paley-Wiener space $B^2_a$ supports supercyclic operators \cite[Theorem 4]{Hai}.

We finish this work by using \eqref{non-expansive} and \eqref{inequality} to characterize all composition operator that are absolutely Ces\`aro bounded on the Paley-Wiener space $B^2_a.$

\begin{thm}\label{abs}
Let $\phi(z)=cz+d$ where $c\in \mathbb{R}$ with $0<|c|\leq1$ and $d\in \mathbb{C}$. Then $C_{\phi}$ is absolutely Ces\`aro bounded on $ B^2_a$ if and only if $c=-1$ or $c=1$ and $d\in\mathbb{R}.$ 
\end{thm}

\begin{proof} We have already seen that if $c = -1,$ then $\phi^n$ is the identity map or the symbol $\phi$ itself. Hence, in both cases, $\|C_{\phi}^n\|\leq e^{|\mathrm{Im}(d)|a}$ for all $n$ (see \eqref{estimates}). This inequality also holds if $c=1$ and $d\in \mathbb{R},$ since $C_{\phi}$ is unitary (see Proposition \ref{normal-unitary}). Then given $f\in B^2_a$ and $n\in \mathbb{N},$ we have
\begin{align*}
 \dfrac{1}{n}\sum_{j=1}^{n}\|C^j_{\phi}f\|_{a}\leq \frac{1}{n}\sum_{j=1}^ne^{|\mathrm{Im}(d)|a}\|f\|_a=e^{|\mathrm{Im}(d)|a}\|f\|_a
\end{align*}
which shows us that $C_{\phi}$ is absolutely Ces\`aro bounded. Now we  consider $0<|c|<1.$ In this case, the sequence formed by the numbers $\frac{1}{n\sqrt{|c|^n}}$ is unbounded. Moreover,  by Proposition \ref{second estimates} there exists $\delta > 0$ such that $\|C^{n}_{\phi}k_1\|_{a} \geq (\delta/\sqrt{|c|^n})\|k_1\|_{a}$ for all sufficiently large $n \in \mathbb{N}.$ Taking sufficiently large $n,$ we get
\begin{align*}
\frac{1}{n}\sum_{j=1}^{n}\|C^{j}_{\phi}k_1\|_{a}\geq \frac{1}{n}\|C_{\phi}^nk_1\|_{a}\geq  \frac{\delta}{n\sqrt{|c|^n}}\|k_1\|_{a}
\end{align*}
which implies that $C_{\phi}$ is not an absolutely Cesàro bounded operator. Finally, we consider $c=1$ and $d\in\mathbb{C}\setminus\mathbb{R}.$ Let $F_2\in L^2[-a,a]$ be the a function given in Lemma \ref{witness lemma} and 
$U$ as in Proposition \ref{third estimate}, then $UF_2$ is a non-zero function in $B^2_a$ and for all $n$ we have
\begin{align*}
\|C_{\phi}^nUF_2\|_a=\|UM_{e^{idt}}^nF_2\|_{a}=\|M_{e^{idt}}^nF_2\|_{L^2[-a,a]}=\sqrt{\frac{e^{2n\,|\mathrm{Im}(d)|\,a}-1}{2n\,|\mathrm{Im}(d)|}}.
\end{align*}
Since the numbers $\sqrt{\frac{e^{2n\,|\mathrm{Im}(d)|\,a}-1}{2n^3\,|\mathrm{Im}(d)|}}$ form a unbounded sequence and 
\begin{align*}
\frac{1}{n}\sum_{j=1}^{n}\|C_{\phi}^{j}UF_2\|_{a}\geq \frac{1}{n}\|C_{\phi}^{n}F_2\|_{a}\geq\frac{1}{n}\sqrt{\frac{e^{2n\,|\mathrm{Im}(d)|\,a}-1}{2n\,|\mathrm{Im}(d)|}}=\sqrt{\frac{e^{2n\,|\mathrm{Im}(d)|\,a}-1}{2n^3\,|\mathrm{Im}(d)|}},
\end{align*}
it follows that $C_{\phi}$ is not absolutely Ces\`aro bounded. 
\end{proof}

\printbibliography


\end{document}